\documentclass[psamsfonts,reqno]{amsart}
\usepackage{pstricks}

\usepackage{amssymb,xy,eucal}

\usepackage{mathrsfs}

\theoremstyle{plain}

\newtheorem{lemma}{Lemma}[section]
\newtheorem{theorem}[lemma]{Theorem}

\newtheorem*{stat}{\name}
\newcommand{\name}{testing}

\theoremstyle{definition}
\newtheorem{definition}[lemma]{Definition}
\newtheorem{example}[lemma]{Example}
\newtheorem{problem}{Problem}

\theoremstyle{remark}
\newtheorem{remark}[lemma]{Remark}

\newtheorem*{remark*}{Remark}

\newcommand{\qedc}{{\qed}~{\rm Claim~{\theclaim}.}}
\newcommand{\qedsc}{{\qed}~{\rm Claim.}}

\numberwithin{equation}{section}

\newcommand{\set}[1]{\{#1\}}
\newcommand{\setm}[2]{\set{#1\mid#2}}

\newcommand{\fami}[2]{(#1)_{#2}}

\newcommand{\upw}{\mathbin{\uparrow}}

\newcommand{\cA}{\mathcal{A}}

\newcommand{\cC}{\mathcal{C}}
\newcommand{\cD}{\mathcal{D}}
\newcommand{\cF}{\mathcal{F}}

\newcommand{\cK}{\mathcal{K}}
\newcommand{\cL}{\mathcal{L}}
\newcommand{\cM}{\mathcal{M}}
\newcommand{\cN}{\mathcal{N}}

\newcommand{\cV}{\mathcal{V}}
\newcommand{\cW}{\mathcal{W}}

\newcommand{\cQ}{\mathcal{Q}}

\DeclareMathOperator{\crita}{crit}

\newcommand{\crit}[2]{\crita({{#1};{#2}})}

\DeclareMathOperator{\Con}{Con}

\DeclareMathOperator{\SI}{SI}
\DeclareMathOperator{\Ker}{Ker}
\newcommand{\dv}{\boldsymbol{2}}

\DeclareMathOperator{\M}{M}
\DeclareMathOperator{\length}{length}

\begin{document}

%\markboth{P. Gillibert and M. Plo\v s\v cica}{Congruence FD-maximal varieties of algebras}

\title{Congruence FD-maximal varieties of algebras}
%\date{\today}

\author[P.~Gillibert]{Pierre Gillibert}
\address{P.~Gillibert, Charles University in Prague, Faculty of Mathematics and Physics, Department of Algebra, Sokolovsk\'a 83, 186 00 Prague, Czech Republic.}
\address{Centro de \'Algebra da Universidade de Lisboa, Av. Prof. Gama Pinto 2, 1649-003 Lisboa, Portugal.}
\email{gilliber@karlin.mff.cuni.cz, pgillibert@yahoo.fr}
\email{http://www.math.unicaen.fr/\~{}giliberp/}

\author[M.~Plo\v s\v cica]{Miroslav Plo\v s\v cica} %\footnote{Supported by VEGA Grant 2/0194/10.}}
\address{M.~Plo\v s\v cica,
Mathematical Institute, Slovak Academy of Sciences,
Gre\v s\'akova 6, 04001 Ko\v sice, Slovakia.}
\address{Institute of Mathematics, \v Saf\'arik's University\\
Jesenn\'a 5, 04154 Ko\v sice, Slovakia}
\email{miroslav.ploscica@upjs.sk}

\thanks{Supported by VEGA Grant 2/0194/10 and the institutional grant MSM 0021620839.}

%06B10 : Ideals, congruence relations
%08A30 : Subalgebras, congruence relations

\begin{abstract}
We study the class of finite lattices that are isomorphic to the congruence lattices of algebras from a given finitely generated congruence-distributive variety. If this class is as large as allowed by an obvious necessary condition, the variety is called congruence FD-maximal. The main results of this paper characterize some special congruence FD-maximal varieties.
\end{abstract}

\maketitle

%\maketitle
\section{Introduction}

The study of congruence lattices is one of the central topics in universal algebra. For a class~$\cK$ of algebras we denote by $\Con\cK$ the class of all lattices isomorphic to $\Con A$ (the congruence lattice of an algebra $A$) for some $A\in\cK$. There are many papers investigating $\Con\cK$ for various classes~$\cK$. However, the full description of $\Con\cK$ has proved to be a very difficult (and probably intractable) problem, even for the most common classes of algebras, like groups or lattices. A recent evidence of this is the solution of the Congruence Lattice Problem (CLP) by F. Wehrung~\cite{WE}.

The difficulty in describing $\Con\cK$ leads to the consideration of the following problem.

\begin{problem}\label{pro}
Given the classes~$\cK$ and~$\cL$ of algebras, decide whether $\Con\cK\subseteq\Con\cL$.
\end{problem}

Note that Problem~\ref{pro} is a loosely formulated array of problems. A ``real'' problem which arises from Problem~\ref{pro} is, for example, whether the containment $\Con\cK\subseteq\Con\cL$ is \emph{decidable}, for \emph{finitely generated} varieties~$\cK$ and~$\cL$.

The problem seems more tractable when both~$\cK$ and~$\cL$ are congruence-distributi\-ve varieties (equational classes) of algebras. In this case many partial results are available, as well as several promising strategies, for instance topological (\cite{P1}) or categorical (\cite{PG}).

The knowledge obtained so far shows the importance of critical points in the sense of the following definition. Let $L_c$ denote the set of all compact elements of an algebraic lattice~$L$.

\begin{definition}\textup{(See \cite{PG}.)} Let~$\cK$ and~$\cL$ be classes of algebras. The \emph{critical point of~$\cK$ under~$\cL$}, denoted by $\crit\cK\cL$, is the smallest cardinality of $L_c$ for $L\in\Con\cK\setminus\Con\cL$ (if $\Con\cK\nsubseteq
\Con\cL$) or $\infty$ (if $\Con\cK\subseteq\Con\cL$).
\end{definition}

For most~$\cK$ and~$\cL$, the critical point is finite or countably infinite. Examples of varieties with uncountable $\crit\cK\cL$ has been found in \cite{PG}, \cite{G2} and \cite{P2}. 

So far, most of the work on critical points has been devoted to distinguishing $\Con\cK$ and $\Con\cL$ in the case of an infinite critical point. The methods include using algebraic and topological invariants (\cite{P1}, \cite{P3}, \cite{P4}), refinement properties(\cite{W1}, \cite{PTW}, \cite{P5}) and lifting of diagrams by the $\Con$ functor (\cite{PG}, \cite{G2}). The case where ~$\cK$ and~$\cL$ are both finitely generated varieties of lattices has been particularly studied. For instance, if~$\cK$ is neither contained in~$\cL$ nor its dual then $\crit\cK\cL\le\aleph_2$ (\cite{G3}).

The case of a finite critical point is also interesting. If~$\cK$ is a finitely generated congruence-distributive variety and~$L$ is a finite distributive lattice, then there is an algorithm deciding whether $L\in\Con\cK$. (For instance, see \cite{N5}.) However, this does not provide a description of the class $\Con\cK$ and we have no algorithm deciding whether $\crit\cK\cL$ is finite, even if both~$\cK$ and~$\cL$ are finitely generated and congruence-distributive.

This paper is devoted to the finite version of the Problem~\ref{pro}, that is, to the characterization of the finite members of $\Con\cK$ and to comparing the finite parts of $\Con\cK$ and $\Con\cL$ for different~$\cK$ and~$\cL$. Even this task is difficult and we can present a satisfactory solution only in some very special cases. However, we believe that the ideas contained in this paper can be applied in more general situations and can contribute significantly to the solution of Problem~\ref{pro} (for congruence-distributive and finitely generated varieties~$\cK$ and~$\cL$).

\section{Basic Concepts}

Now we recall basic denotations and facts. We assume familiarity with basic concepts from lattice theory and universal algebra. For all undefined concepts and unreferenced facts we refer to \cite{GG} and \cite{UA}.

If $f\colon\ X\to Y$ is a mapping and $Z\subseteq X$, then $f{\restriction} Z$ denotes the restriction of~$f$ to $Z$. Furthermore, $\Ker f$ (the kernel of~$f$) is the binary relation on $X$ defined by $\Ker f=\setm{(x,y)\in X^2}{f(x)=f(y)}$. If~$f$ is a homomorphism of algebras, then $\Ker f$ is a congruence. Given maps $f,g\colon A\to B$ and $X\subseteq A$, we denote $(f,g)[X]=\setm{(f(x),g(x))}{x\in X}$.

For $\alpha\in\Con A$ and $x\in A$, we denote by $x/\alpha$ the congruence class containing $x$.

We denote by $\SI(\cV)$ the class of all subdirectly irreducible members of a variety~$\cV$.

If $a$ is an element of a lattice (or an ordered set)~$L$, then we denote $\upw a=\setm{b\in L}{b\ge a}$. The smallest and the greatest element of any lattice (if they exist) will be denoted by $0$ and $1$, respectively. The two-element lattice will be denoted by $\dv=\set{0,1}$. An element $a$ of a lattice~$L$ is called \emph{completely meet-irreducible} iff $a=\bigwedge X$ implies that $a\in X$, for every subset $X$ of~$L$. The greatest element of~$L$ is not completely meet-irreducible. Let $\M(L)$ denote the set of all completely meet-irreducible elements of~$L$. Clearly, $\alpha\in\Con A$ is completely meet-irreducible if and only if the quotient algebra $A/\alpha$ is subdirectly irreducible. Hence, we have the following easy fact.

\begin{lemma}\label{11}
Let~$\cV$ be a variety and $L=\Con A$ for some $A\in\cV$. Then for every $x\in \M(L)$, the lattice~$\upw x$ is isomorphic to $\Con T$ for some $T\in \SI(\cV)$.
\end{lemma}

This lemma provides a basic information about congruence lattices of algebras in~$\cV$. It is especially effective in the case of a congruence-distributive variety~$\cV$ and a finite lattice~$L$, because finite distributive lattices are determined uniquely by the ordered sets of their meet-irreducible elements. 

There are varieties for which the necessary condition of Lemma~\ref{11} (for distributive lattices) is also sufficient on the finite level. We call such varieties \emph{congruence FD-maximal} Hence~$\cV$ is congruence FD-maximal if for every finite distributive lattice~$L$ the following conditions are equivalent:
\begin{enumerate}
\item[(i)] $L\in\Con\cV$;
\item[(ii)] $\upw x\in\Con\cV$ for each $x\in\M(L)$.
\end{enumerate}

\begin{remark}
Notice that the equivalence of (i) and (ii) in the definition of congruence FD-maximality, only holds for finite distributive lattices. For example $\cD$, the variety of distributive lattices, is congruence FD-maximal. The finite congruence lattices in $\Con\cD$ are the finite Boolean algebras. The only subdirectly irreducible lattice in $\cD$ is the two-element lattice.

Denote by $M_3$ the lattice of length two with three atoms. It is easy to see that~$\upw x$ is the two-element lattice, for each meet-irreducible $x\in M_3$, which is isomorphic to $\Con\dv$. However $M_3$ does not belong to $\Con\cD$.

There also exists an infinite algebraic distributive lattice~$L$, such that~$\upw x$ is the two-element lattice, for each meet-irreducible $x\in L$, but~$L$ does not belongs to $\Con\cD$.
\end{remark}

\begin{example}
The variety of all lattices is congruence FD-maximal, as each finite distributive lattice is isomorphic to the congruence lattice of a lattice (the first published proof is due to G. Gr{\"a}tzer and E.\,T. Schmidt in \cite{GrSc62}). For similar reason the variety of all modular lattices is congruence FD-maximal (see \cite{S74}).
\end{example}

\section{The limit construction}
The algebras with a prescribed congruence lattices can be often constructed as limits of suitable diagrams. Let us recall the construction from \cite{N5}.

An \emph{ordered diagram of sets} is a triple $(P,\cA,\cF)$, where $P$ is a partially ordered set, $\cA=\fami{A_p}{p\in P}$ is a family of sets indexed by $P$ and $\cF=\fami{f_{pq}}{p\le q\text{ in }P}$ is a family of functions $f_{pq}\colon A_p\to A_q$ such that
\begin{enumerate}
\item[(1)] $f_{pp}$ is the identity map for every $p\in P$;
\item[(2)] $f_{qr}f_{pq}=f_{pr}$ for every $p,q,r\in P$, $p\le q\le r$.
\end{enumerate}
For any ordered diagram of sets we define its {\it limit} as
\[
\lim(P,\cA,\cF)=\setm{a\in\prod_{p\in P}A_p}{a_q=f_{pq}(a_p)\text{ for every }p\le q\text{ in }P}.
\]
Thus, our limit is the limit in the sense of the category theory (applied to the category of sets). Moreover, if $A_p$ are algebras of the same type and $f_{p,q}$ are homomorphisms, then the limit is a subalgebra of the direct product. In universal algebra, this construction is often called \emph{the inverse limit} (see \cite{UG}). 

\begin{definition} \label{def} An ordered diagram of sets
$(P,\cA,\cF)$
is called {\it admissible} if the following conditions are satisfied:
\begin{enumerate}
\item[(i)] for every $p\in P$ and every $u\in A_p$ there exists
$a\in\lim(P,\cA,\cF)$ such that $a_p=u$;
\item[(ii)] for every $p\not\le q$ in $P$, there exist $a,b\in\lim
(P,\cA,\cF)$ such that $a_p=b_p$ and $a_q\ne b_q$.
\end{enumerate}
\end{definition}

\begin{theorem}[{\cite[Theorem~2.4]{N5}}]\label{re1}
Let~$\cV$ be a congruence-distributive variety and let~$L$ be a finite distributive lattice. Set $P=\M(L)$. For every $p\in P$ let $A_p\in\cV$ and for every $p\le q$ in $P$ let $f_{pq}\colon A_p\to A_q$ be a homomorphism such that $(P,\cA,\cF)$ is an admissible ordered diagram of sets (with $\cA=\fami{A_p}{p\in P}$ and $\cF=\fami{f_{pq}}{p\le q\text{ in }P}$ and, moreover,
\begin{enumerate}
\item[(*)] for all $p\in P$, the sets $\setm{\Ker(f_{pq})}{ q\ge p\text{ in } P}$
and $M(\Con A_p)$ coincide.
\end{enumerate}
Then $\lim(P,\cA,\cF)$ is an algebra whose congruence lattice is isomorphic to~$L$.
\end{theorem}

We demonstrate the use of this theorem in the following easy case.

\begin{theorem}
Let~$\cV$ be a congruence-distributive variety such that $\Con C$ is a finite chain for every $C\in\SI(\cV)$. Suppose that $n=\max\setm{\length(\Con C)}{C\in\SI(\cV)}$ exists. Let~$L$ be a finite distributive lattice. The following conditions are equivalent.
\begin{enumerate}
\item[(i)] $L\in\Con\cV$;
\item[(ii)] For every $x\in \M(L)$, the poset~$\upw x$ is a chain of length at most $n$.
\end{enumerate}\label{chains}
\end{theorem}

\begin{proof}
The necessity of (ii) follows directly from Lemma~\ref{11}. Conversely, suppose that~$L$ satisfies
(ii). We construct an algebra $A\in\cV$ with $\Con A$ isomorphic to~$L$.

There exist $C\in\cV$ whose congruence lattice is an $(n+1)$-element chain
$\alpha_0>\alpha_1>\dots>\alpha_n$. For every $i=1,\dots,n$ the quotient algebra $C_i=C/\alpha_i$
is subdirectly irreducible and $\Con C_i$ is the $(i+1)$-element chain. Further, for every
$j\le i$ we have a natural (surjective) homomorphism $g_{ij}\colon C_i\to C_j$.

We define the following diagram indexed by the ordered set $P=\M(L)$. For every $p\in P$,
$\upw p$ is a chain. We denote the cardinality of this chain by $i(p)$
and set $A_p=C_{i(p)}$. For every $p\le q$ in $P$, we have $i(p)\ge i(q)$; we set 
$f_{pq}=g_{i(p),i(q)}$. Let $A$ be the limit of this diagram.
It remains to check that the assumptions of Theorem~\ref{re1}
are satisfied.

For every $x\in C$ consider $(x/\alpha_{i(p)})_{p\in P}\in\prod_{p\in p}A_p$. It 
is clear that this is an element of $A$. All such elements show that all projections
$A\to A_p$ are surjective.

To prove the second admissibility condition, let $p,q\in P$ such that $p\not\le q$. Let $j=i(q)$. 
Since $j>0$ and $\alpha _j\subsetneq\alpha_{j-1}$, it is possible to choose $x,y\in C$ with
$x/\alpha_j\ne y/\alpha_j$ and
$x/\alpha_{j-1}=y/\alpha_{j-1}$. We define elements $a=(a_t),b=(b_t)\in\prod_{t\in P}A_t$ by

\[
a_t=\begin{cases}
      x/\alpha_{i(t)} &\text{if $t\le q$}\\
      y/\alpha_{i(t)} &\text{if $t\not\le q$,}
\end{cases}
\]
\[
b_t=y/\alpha_{i(t)}\, ,\quad\text{for every $t\in P$.}
\]
It is easy to see that $a,b\in A$, $a_p=b_p$, $a_q\ne b_q$. Indeed, the only nontrivial point is to show that $f_{rs}(a_r)=a_s$ when $r\le q$, $s\not\le q$, $r\le s$. Since~$\upw r$ is a chain, we have $q<s$, so $j=i(q)>i(s)$, and $f_{rs}(a_r)=f_{qs}(f_{rq}(a_r))=f_{qs}(a_q)=g_{j,i(s)}(x/\alpha_j)=g_{j-1,i(s)}(x/\alpha_{j-1})=g_{j-1,i(s)}(y/\alpha_{j-1})=y/\alpha_{i(s)}=a_s$.

Thus, our diagram is admissible. Finally, we check the condition (*). For every $p\in P$ we have $\M(\Con A_p)=\setm{\alpha_j/\alpha_{i(p)}}{j=1,\dots,i(p)}$, which is exactly the set of the kernels of homomorphisms $f_{pq}=g_{ij}$.
\end{proof}

The assumptions of Theorem~\ref{chains} are satisfied for instance for any finitely generated variety of modular lattices. Every finite subdirectly irreducible modular lattice is simple, hence its congruence lattice is the two-element chain.

Another example is the variety of Stone algebras, generated by the three-element pseudocomplemented lattice $S=(\set{0,b,1};\wedge,\vee,0,1,{}')$, with $0<b<1$. The congruence lattice of $S$ is the three-element chain. The only other irreducible member in this variety is the two-element Boolean algebra, which is simple.

So, in the terminology from the introduction, any finitely generated variety of modular lattices, and the variety of Stone algebras are all congruence FD-maximal.

\section{The case of V-varieties}
In this section we discuss the simplest case not covered by Theorem~\ref{chains}. We consider finitely generated congruence-distributive varieties, in which every subdirectly irreducible algebra is simple, or has its congruence lattice isomorphic to the lattice $V$ depicted in Figure~\ref{F:V}.

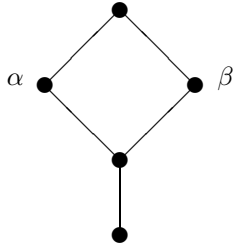
\begin{figure}
\unitlength 1mm
\begin{picture}(120,40)(0,10)
\put(35,30){$\alpha$}
\put(63,30){$\beta$}
\put(50,10){\line(0,1){10}}
\put(50,20){\line(-1,1){10}}
\put(50,20){\line(1,1){10}}
\put(40,30){\line(1,1){10}}
\put(60,30){\line(-1,1){10}}
\put(50,10){\circle*{2}}
\put(50,20){\circle*{2}}
\put(40,30){\circle*{2}}
\put(60,30){\circle*{2}}
\put(50,40){\circle*{2}}
\end{picture}
\caption{The lattice $V$.}\label{F:V}
\end{figure}

In the sequel, such varieties will be called \emph{V-varieties}. The similarity type of a V-variety is not assumed finite. If~$\cV$ is a V-variety then, by Lemma~\ref{11}, every finite distributive $L\in\Con\cV$ satisfies
\begin{enumerate}
\item[(**)] the ordered set $\M(L)$ is a disjoint union of two antichains $N\cup D$ such that for every $n\in N$ there are exactly two $d\in D$ with $n<d$.
\end{enumerate}
Indeed, if $L=\Con A$, $A\in\cV$, and $\alpha\in \M(L)$, then $\alpha\in N$ when $\Con(A/\alpha)\cong V$ and $\alpha\in D$ when $A/\alpha$ is simple. It follows that for a V-variety~$\cV$, if every finite distributive lattice satisfying (**) belongs to $\Con\cV$, then~$\cV$ is congruence FD-maximal.

Consider the smallest nonmodular lattice variety $\cN_5$ generated by the lattice $N_5$ depicted in Figure~\ref{F:smalllat}. By \cite[Theorem 3.1]{N5} the condition (**) is equivalent to saying that~$L$ belongs to $\Con\cN_5$. Therefore $\cN_5$ is congruence FD-maximal. It follows that a V-variety $\cV$ is congruence FD-maximal if and only if each finite lattice in $\Con\cN_5$ belongs to $\Con\cV$.

\begin{figure}[htb]
\setlength{\unitlength}{1mm}
\begin{picture}(45,60)(-5,-5)
\put(20,50){$N_5$}
\put(20,0){\line(-1,1){20}}
\put(20,0){\line(1,1){15}}
\put(35,15){\line(0,1){10}}
\put(0,20){\line(1,1){20}}
\put(35,25){\line(-1,1){15}}
\put(20,0){\circle*{2}}
\put(0,20){\circle*{2}}
\put(35,15){\circle*{2}}
\put(35,25){\circle*{2}}
\put(20,40){\circle*{2}}
\put(19,-4){$0$}
\put(19,42){$1$}
\put(-5,19){$b$}
\put(38,14){$a$}
\put(38,24){$c$}
\end{picture}
\begin{picture}(35,60)(-5,-5)
\put(15,50){$L_1$}
\put(15,0){\line(-1,1){15}}
\put(15,0){\line(0,1){15}}
\put(15,0){\line(1,1){15}}

\put(0,15){\line(0,1){15}}
\put(15,15){\line(-1,1){15}}
\put(15,15){\line(1,1){15}}
\put(30,15){\line(0,1){15}}

\put(30,30){\line(-1,1){15}}
\put(0,30){\line(1,1){15}}

\put(15,0){\circle*{2}}
\put(0,15){\circle*{2}}
\put(0,30){\circle*{2}}
\put(15,15){\circle*{2}}
\put(30,15){\circle*{2}}
\put(30,30){\circle*{2}}
\put(15,45){\circle*{2}}
\end{picture}
\begin{picture}(35,60)(-5,-5)
\put(15,50){$L_2$}
\put(15,0){\line(-1,1){15}}
\put(15,0){\line(1,1){15}}

\put(0,15){\line(0,1){15}}
\put(0,15){\line(1,1){15}}

\put(30,15){\line(0,1){15}}
\put(30,15){\line(-1,1){15}}

\put(30,30){\line(-1,1){15}}
\put(0,30){\line(1,1){15}}
\put(15,30){\line(0,1){15}}

\put(15,0){\circle*{2}}
\put(0,15){\circle*{2}}
\put(0,30){\circle*{2}}
\put(15,30){\circle*{2}}
\put(30,15){\circle*{2}}
\put(30,30){\circle*{2}}
\put(15,45){\circle*{2}}
\end{picture}
\caption{Small lattices generating V-varieties.}\label{F:smalllat}
\end{figure}
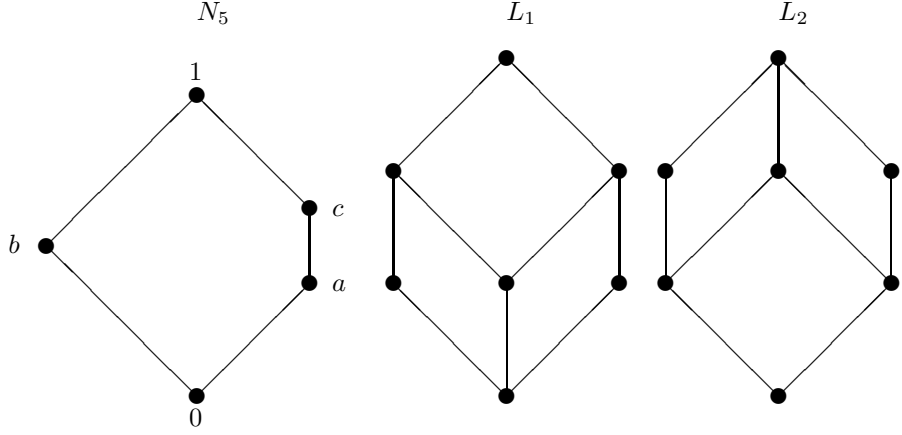

Denote by $\cL_1$ the variety of lattices generated by $L_1$, and by $\cL_2$ the variety of lattices generated by $L_2$. Given any finitely generated variety of modular lattices $\cM$, the varieties of lattices $\cM\vee\cN_5$, $\cM\vee\cL_1$, and $\cM\vee\cL_2$ are congruence FD-maximal V-varieties.

Let $E$ be a subset of $B\times B$ for some set $B$. Further, let $X$ be a set and let $\cF$ be a set of functions $X\to B$. We say that $\cF$ is $E$-\emph{compatible}
if there exists a linear order~$\sqsubset$ on $\cF$ such that $(f,g)[X]=E$ for all $f\sqsubset g$ in $\cF$. The set $\cF$ is called \emph{strongly} $E$-compatible if, in addition, $\Ker f\nsupseteq\bigcap\setm{\Ker g}{g\in\cF\setminus\set{f}}$ for every $f\in\cF$.

\begin{lemma}
Let $B$ be a finite set, let $E\subseteq B\times B$ and suppose that there is $(u,v)\in E$ with $u\ne v$.
Then the following condition are equivalent.
\begin{enumerate}
\item[(i)] There exist arbitrarily large finite $E$-compatible sets of functions.
\item[(ii)] There exist arbitrarily large finite strongly $E$-compatible sets of functions.
\item[(iii)] For every $(a,b)\in E$ there are $x,y,z\in B$ such that $(x,x)$, $(y,y)$, $(z,z)$, $(x,y)$, $(x,z)$, $(y,z)$, $(x,a)$, $(x,b)$, $(a,y)$, $(y,b)$, $(a,z)$, $(b,z)\in E$.
\end{enumerate}\label{xyz}
\end{lemma}

\begin{proof}
Suppose first that (i) holds. Let $(a,b)$ in $E$. Let $n=3k+5$, where $k$ is the cardinality of $B$. By our assumption, there exists a set $X$ and an $E$-compatible set $\cF=\set{f_1,\dots,f_n}$ of functions $X\to B$ such that $(f_i,f_j)[X]=E$ for all $i<j$ in $\set{1,\dots,n}$. Especially, set $i=k+2$, $j=2k+4$. Since $(a,b)\in E=(f_i,f_j)[X]$, there exists $t\in X$ such that $f_i(t)=a$ and $f_j(t)=b$.

By the pigeonhole principle, there are $l<m$ in $\set{1,\dots,k+1}$ such that $f_l(t)=f_m(t)$. We put $x=f_m(t)$, so $(x,x)=(f_l(t),f_m(t))\in E$, $(x,a)=(f_m(t),f_i(t))\in E$, and $(x,b)=(f_m(t),f_j(t))\in E$.

Further, there are $p<q$ in $\set{k+3,\dots,2k+3}$ such that $f_p(t)=f_q(t)$. We put $y=f_p(t)$. Similarly as above we can check that $(y,y),(a,y),(y,b),(x,y)\in E$.

Finally, there are $r<s$ in $\set{2k+5,\dots,3k+5}$ such that $f_r(t)=f_s(t)$. We put $z=f_r(t)$ and obtain that $(z,z),(x,z),(y,z),(a,z),(b,z)\in E$. This completes the proof of the implication (i)$\Longrightarrow$(iii).

Suppose now that (iii) holds. Let $k$ be a natural number. We need to construct a set $X$ and a strongly $E$-compatible set $\set{f_1,\dots,f_k}$ of functions $X\to B$.

Given $(a,b)\in E$, we fix $x_{ab},y_{ab},z_{ab}\in B$ that satisfy (iii). We set
\[
H=\set{1/2,1,3/2,2,5/2,\dots,k, k+1/2}
\]
and
\[
U=\setm{(i,j)}{i<j\text{ in }H}
\]
Put $X=E\times U$. Given $m\in\set{1,\dots,k}$, $(a,b)\in E$, and $(i,j)\in U$, we set
\[
f_m(a,b,i,j) =
\begin{cases}
x_{ab} & \text{if $m<i$}\\
a & \text{if $m=i$}\\
y_{ab} & \text{if $i<m<j$}\\
b & \text{if $m=j$}\\
z_{ab} & \text{if $m>j$.}
\end{cases}
\]
This defines a map $f_m\colon X\to B$ for each $m\in\set{1,\dots,k}$. Let $m<n$ in $\set{1,\dots,k}$. Given $(a,b)\in E$ and $(i,j)\in U$, it is easy to see that $(f_m(a,b,i,j),f_n(a,b,i,j))\in E$, therefore the containment $(f_m,f_n)[X]\subseteq E$ holds. Conversely let $(a,b)\in E$, notice that $f_m(a,b,m,n)=a$ and $f_n(a,b,m,n)=b$, it follows that $(a,b)\in (f_m,f_n)[X]$. We deduce that $(f_m,f_n)[X]=E$ for all $n<m$ in $\set{1,\dots,k}$. Therefore the set $\set{f_1,\dots,f_k}$ is $E$-compatible.

Further, choose $(a,b)\in E$ with $a\ne b$. Then for every $m\in\set{1,\dots,k}$ we have $f_m(a,b,m-1/2,m)=b\ne a=f_m(a,b,m,m+1/2)$, so the pair $((a,b,m-1/2,m),(a,b,m,m+1/2))$ does not belong to $\Ker f_m$. On the other hand, it is easy to see that this pair belongs to $\Ker f_n$ for every $n\ne m$. Thus, our set is strongly $E$-compatible and we have proved the implication (iii)$\Longrightarrow$(ii). Since the implication (ii)$\Longrightarrow$(i) is trivial, the proof is complete.
\end{proof}

\begin{theorem}\label{T:caractmaxVvar}
Let~$\cV$ be a V-variety. The following conditions are equivalent.
\begin{enumerate} 
\item[(1)] $\cV$ is congruence FD-maximal.
\item[(2)] There exist $B,C\in\cV$ and surjective homomorphisms $h_0,h_1\colon C\to B$ such that
 \begin{enumerate}
 \item[(i)] $B$ is simple, $\Con C\cong V$;
 \item[(ii)] $\Ker(h_0)\ne\Ker(h_1)$;
 \item[(iii)] there are arbitrarily large $E$-compatible sets of functions for the set
 $E=(h_0,h_1)[C]=\setm{(h_0(c),h_1(c))}{c\in C}\subseteq B\times B$.
 \end{enumerate}
\end{enumerate}\label{main}
\end{theorem} 

\begin{remark}
Let us note, in the context of Theorem~\ref{T:caractmaxVvar} that $\Ker(h_0)$ and $\Ker(h_1)$ correspond to congruences $\alpha,\beta\in V$. Also note that, $\cV$ being a finitely generated congruence-distributive variety, the algebras $B$ and $C$ are both finite.
\end{remark}

\begin{proof}[{Proof of \textup{Theorem~\ref{T:caractmaxVvar}}}]
Suppose that (1) holds. As~$\cV$ is congruence-distributive and finitely generated, it can only contain finitely many (up to isomorphism) algebras $C_1,\dots,C_m$ whose congruence lattice is isomorphic to $V$, and finitely many (up to isomorphism) simple algebras $B_1,\dots,B_l$. Moreover the algebras $C_1,\dots,C_m$ and $B_1,\dots,B_l$ are all finite.

Let $\cQ$ be the set of all ordered pairs $Q=(h_0,h_1)$ of surjective homomorphisms $h_0\colon C_i\to B_{i_0}$, $h_1\colon C_i\to B_{i_1}$ for some $i,i_0,i_1$, such that $\Ker(h_0)\ne\Ker(h_1)$. For every such $Q=(h_0,h_1)$ we denote $E_Q=(h_0,h_1)[C_i]$.

Clearly, every $C_i$ occurs in some $Q$. Notice that some $C_i$ can occur more than twice. This is because the algebras $B_i$ can have proper automorphisms and then the homomorphisms $h_0$, $h_1$ are not determined uniquely by their kernels. Denote by $r$ the cardinality of $\cQ$. 

For contradiction, suppose that (2) is not valid, that is, for every $Q=(h_0,h_1)\in\cQ$ either $h_0$ and $h_1$ have distinct ranges or the size of a set of $E_Q$-compatible functions is bounded, say by a number $n_Q$. Choose a natural number $n$ with $n>2$ and $n>n_Q$ for every defined $n_Q$. By the Ramsey's Theorem, we can choose $k$ such that every complete graph with $k$ vertices and edges colored by $r$ (the cardinality of $Q$) colors has a full subgraph with $n$ vertices, in which every edge has the same color.

Consider the finite distributive lattice~$L$ such that $\M(L)$ is the disjoint union of two discrete sets $D=\set{d_1,\dots,d_k}$ and $N=\setm{x_{ij}}{1\le i<j\le k}$, with the $d_i$ (resp., the $x_{ij}$) pairwise distinct, and the only relations $x_{ij}<d_i,d_j$ for $1\le i<j\le k$.

According to our assumption,~$L$ is representable as $\Con A$ for some finite $A\in\cV$. Every quotient $A/x_{ij}$ must be isomorphic to some $C_t$ and every quotient $A/d_i$ is isomorphic to some $B_u$. Let us fix the isomorphisms $g_i\colon A/d_i\to B_u$ and $g_{ij}\colon A/x_{ij}\to C_t$. Further, let $p_{ij}\colon A/x_{ij}\to A/d_i$ and $q_{ij}\colon A/x_{ij}\to A/d_j$ be the natural projections. Then $(g_ip_{ij}g_{ij}^{-1},g_jq_{ij}g_{ij}^{-1})$ is an element of $\cQ$.

Let us consider the complete graph with vertices $d_1,\dots,d_k$. To every edge $\set{d_i,d_j}$, $i<j$, we have assigned an element of $\cQ$. By the definition of $k$, there is a $n$-element subset $I\subseteq\set{1,\dots,k}$ such that the same $Q=(h_0,h_1)\in\cQ$ is assigned to every $i<j$ in $I$. To simplify the notation, assume that $I=\set{1,\dots,n}$.

The algebras $A/x_{ij}$ ($i<j$ in $I$) are all isomorphic to the same algebra $C\in\set{C_1,\dots,C_m}$. Since $n\ge 3$, the algebras $A/d_i$ ($i\in I$) must be isomorphic to the same algebra $B\in\set{B_1,\dots,B_l}$, hence both $h_0$ and $h_1$ have range $B$. Therefore, the set $E_Q\subseteq B^2$ is defined. For every $i\in I$ let $f_i$ be the natural projection $A\to A/d_i$ composed with the isomorphism $g_i\colon A/d_i\to B$. The maps $f_i$ are pairwise distinct (they have different kernels) and we claim that the family $\setm{f_i}{i\in I}$ is $E_Q$-compatible. Let $i<j$ in $I$, hence $h_0=g_ip_{ij}g_{ij}^{-1}$ and $h_1=g_jq_{ij}g_{ij}^{-1}$. For every $a\in A$ we have $(f_i(a),f_j(a))=(h_0(b),h_1(b))\in E_Q$, where $b=g_{ij}(a_{x_{ij}})\in C$. Since every element of $E_Q$ is of this form, we obtain that $(f_i,f_j)[A]=\setm{(f_i(a),f_j(a))}{a\in A}=E_Q$.

So, we have proved that the family $\setm{f_i}{i\in I}$ is $E_Q$-compatible. This is a contradiction with the inequality $n>n_Q$.

Conversely, suppose that (2) holds. Let~$L$ be a finite distributive lattice with the property that the ordered set $\M(L)$ consists of two antichains $N$ and $D$ such that for every $n\in N$ there are exactly two $d_1,d_2\in D$ with $n<d_1,d_2$. We construct a finite algebra $A\in \cV$ such that $\Con A$ is isomorphic to~$L$.

According to our assumption, there are algebras $B,C\in \cV$ and homomorphisms $h_0,h_1\colon C\to B$ satisfying (i), (ii), and (iii). There is a set $X$ and an $E$-compatible family of functions $X\to B$ of cardinality equal to the cardinality of $D$. To simplify the notation, we identify these functions with elements of $D$. Because of Lemma~\ref{xyz} we can assume that the family is strongly $E$-compatible. We fix a linear order $\sqsubset$ on $D$ such that $(d,e)[X]=E$ for all $d\sqsubset e$ in $D$, such order exists by the definition of $E$-compatibility.

We use the limit construction described in the previous section, with $\M(L)$ as the index set. We put
\[
A_p= \begin{cases}
     B & \text{if $p\in D$}\\
     C & \text{if $p\in N$.}
     \end{cases}
\]

For every $n\in N$ there are $d\sqsubset e$ in $D$, with $n<d,e$. We set $f_{nd}=h_0$ and $f_{ne}=h_1$. Let $A$ be the limit of the above defined diagram.

Our diagram clearly satisfies the condition (*) of Theorem~\ref{re1}. We need to check the admissibility. Given $x\in X$, we denote by $U(x)$ the set of all $y\in\prod_{p\in \M(L)}A_p$ such that $y_d=d(x)$ for all $d\in D$ and $(h_0(y_n),h_1(y_n))=(d(x),e(x))$ for all $n\in N$, where $d\sqsubset e$ in $D$ with $n<d,e$.

Let $y\in U(x)$, let $n\in N$, let $d\sqsubset e$ in $D$ such that $n<d,e$. Notice that $f_{nd}(y_n)=h_0(y_n)=d(x)=y_d$ and $f_{ne}(y_n)=h_1(y_n)=e(x)=y_e$. It follows easily that $y$ belongs to $A$. Therefore $U(x)\subseteq A$. Set $U=\bigcup_{x\in X}U(x)$, note that $U\subseteq A$.

Let $x\in X$. Given $d\in D$, set $y_d=d(x)$. Let $n\in N$ and $d\sqsubset e$ in $D$ such that $n<d,e$, notice that $(h_0,h_1)[X]=E=(d,e)[C]$, therefore there is $y_n\in C$ such that $(h_0(x),h_1(x))=(d(y_n),e(y_n))$. This defines an element $y\in U(x)$, therefore $U(x)$ is not empty.

Let $c_0\in C$, let $n_0\in N$, let $d_0\sqsubset e_0$ in $D$ such that $n_0<d_0,e_0$. As $(h_0,h_1)[C]=E=(d_0,e_0)[X]$, there is $x\in X$ such that $(d_0(x),e_0(x))=(h_0(c_0),h_1(c_0))$. With a construction similar to the previous one we obtain $y\in U(x)$, moreover we can choose $y$ such that $y_{n_0}=c_0$. Therefore, for all $c\in C$, for all $n\in N$, there is $y\in U$ such that $y_n=c$. With a similar argument we obtain that, for all $b\in B$, for all $d\in D$, there is $y\in U$ such that $y_d=b$.

This shows~Definition~\ref{def}(i). To show the second admissibility condition, let $p\nleq q$ in $\M(L)$. We distinguish several cases.

\begin{enumerate}
\item[a)] Let $p,q\in D$. The strong compatibility implies that $\Ker p\nsubseteq\Ker q$, so there is 
$(x,y)\in\Ker p\setminus\Ker q$. Pick $x^*\in U(x)$ and $y^*\in U(y)$, then $x^*_p=y^*_p$ and $x^*_q\ne y^*_q$.

\item[b)] Let $p\in N$, $q\in D$. There are $d\sqsubset e$ in $D$ with $p<d,e$. Clearly, $q\notin\set{d,e}$ and by the strong compatibility there exists $(x,y)\in\Ker d\cap\Ker e$, with $(x,y)\notin\Ker q$. Then $(d(x),e(x))=(d(y),e(y))$, so there are $x^*\in U(x)$ and $y^*\in U(y)$ such that $x^*_p=y^*_p$. On the other hand, $(x,y)\notin\Ker q$ means that $x^*_q\ne y^*_q$.

\item[c)] Finally, let $q\in N$. Since the algebra $A_q=C$ is subdirectly irreducible, there
are elements $u,v\in A_q$ such that $u\ne v$ and the pair $(u,v)$ belongs to the monolith of $A_q$.
There is $y\in U$ with $y_q=u$. Define an element $t\in\prod_{r\in \M(L)}A_r$ by
\[
t_r=
\begin{cases}
v   & \text{if $r=q$}\\
y_r & \text{otherwise.}
\end{cases}
\]
Then $t\in A$. Indeed, let $r,s\in \M(L)$, $r<s$. The equality $f_{rs}(t_r)=t_s$ follows from $y\in A$,
except for the case $r=q$. However, $f_{qs}$ is either $h_0$ or $h_1$, and the pair
$(u,v)$ belongs to the kernels of both, so $f_{qs}(t_q)=f_{qs}(v)=f_{qs}(u)=f_{qs}(y_q)=y_s=t_s$.

Thus, $t\in A$. Clearly, $t_q\ne y_q$ and $t_p=y_p$. The proof is complete.
\end{enumerate}
\end{proof}

Theorem~\ref{main} in connection with Lemma~\ref{xyz} enables to decide whether a given V-variety is congruence FD-maximal. As an example, consider the variety $\cN_5$. The algebra $N_5$ has congruence lattice isomorphic to $V$, both its simple quotients are isomorphic to $\set{0,1}$, and the homomorphisms $h_0,h_1\colon N_5\to\set{0,1}$ are given by $h_0(0)=h_0(a)=h_0(c)=0$, $h_0(b)=h_0(1)=1$, $h_1(b)=h_1(0)=0$, $h_1(a)=h_1(c)=h_1(1)=1$. Hence, $E=\set{(0,0),(0,1),(1,0),(1,1)}$ and it is easy to see that~\ref{xyz}(iii) is satisfied. (Indeed,~\ref{xyz}(iii) always holds for reflexive $E$.) As seen at the beginning of this section, we already knew, from \cite[Theorem 3.1]{N5}, that $\cN_5$ is congruence FD-maximal.

\section{Congruence non FD-maximal examples}\label{S:nmax}

It is not difficult to construct a V-variety which is not congruence FD-maximal. For instance, consider the lattice $N_5$ endowed with the additional unary operation~$f$ with $f(0)=f(b)=b$, $f(a)=f(c)=f(1)=1$. (See the picture of $N_5$ in Figure~\ref{F:smalllat}) This algebra has the same congruences as the lattice $N_5$, so its congruence lattice is isomorphic to $V$. One of the simple quotient of this algebra is isomorphic to the lattice $\set{0,1}$ with the identity as additional operation. The other simple quotient is the lattice $\set{0,1}$ with the map permuting $0$ and $1$ as additional operation. The two simple quotients of this algebra are not isomorphic. By Theorem~\ref{main}, the variety~$\cV$ generated by this algebra is not congruence FD-maximal. (Of course, a detailed proof of this fact requires checking that~$\cV$ does not contain any other algebra with congruence lattice isomorphic to $V$, but this is an easy consequence of J\'onsson's Lemma.)

Now we present a more sophisticated example. Let $C$ be the lattice depicted in Figure~\ref{F:Valgebra} with two additional unary operations~$f$ and~$g$.

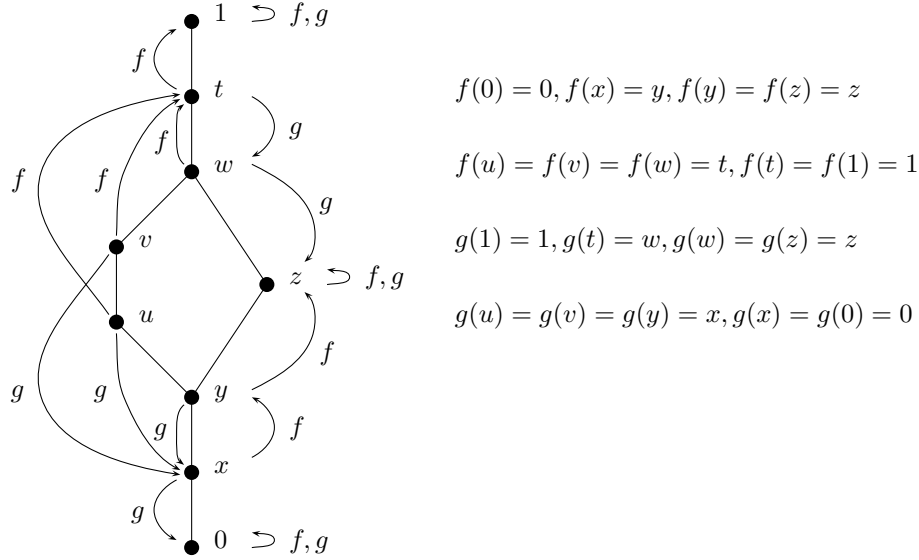
\begin{figure}[htb]
\begin{pspicture}(0.5,-1)(12,8)
\pslinewidth0.4pt
\pscircle*[linecolor=black](3,0){0.1}
\pscircle*[linecolor=black](3,1){0.1}
\pscircle*[linecolor=black](3,2){0.1}

\pscircle*[linecolor=black](2,3){0.1}
\pscircle*[linecolor=black](4,3.5){0.1}
\pscircle*[linecolor=black](2,4){0.1}

\pscircle*[linecolor=black](3,5){0.1}
\pscircle*[linecolor=black](3,6){0.1}
\pscircle*[linecolor=black](3,7){0.1}
\psline(3,0)(3,2)

\psline(3,2)(2,3)
\psline(2,3)(2,4)
\psline(2,4)(3,5)

\psline(3,2)(4,3.5)
\psline(4,3.5)(3,5)

\psline(3,5)(3,7)

\put(3.3,0){$0$}
\put(3.3,1){$x$}
\put(3.3,2){$y$}

\put(2.3,3){$u$}
\put(2.3,4){$v$}
\put(4.3,3.5){$z$}

\put(3.3,5){$w$}
\put(3.3,6){$t$}
\put(3.3,7){$1$}

%de 0 à 0 pour f et g
\psecurve{->}(3.5,0)(3.8,0)(4.1,0.1)(3.8,0.2)(3.5,0.2)
\put(4.3,0){$f,g$}
%de z à z pour f et g
\psecurve{->}(4.5,3.5)(4.8,3.5)(5.1,3.6)(4.8,3.7)(4.5,3.7)
\put(5.3,3.5){$f,g$}
%de 1 à 1 pour f et g
\psecurve{->}(3.5,7)(3.8,7)(4.1,7.1)(3.8,7.2)(3.5,7.2)
\put(4.3,7){$f,g$}

%de x à y pour f
\psecurve{->}(3.5,1.2)(3.8,1.2)(4.1,1.6)(3.8,2)(3.5,2)
\put(4.3,1.5){$f$}
%de y à z pour f
\psecurve{->}(3.3,2)(3.8,2.1) (4.6,2.8) (4.5,3.4)(3.5,3.5)
\put(4.7,2.5){$f$}

%de x à 0 pour g
\psecurve{->}(3.1,0.9)(2.8,0.9)(2.5,0.5)(2.8,0.1)(3.1,0.1)
\put(2.2,0.4){$g$}

%de y à x pour g
\psecurve{->}(3.1,1.9)(2.9,1.9)(2.8,1.5)(2.9,1.1)(3.1,1.1)
\put(2.5,1.5){$g$}

%de u à x pour g
\psecurve{->}(2,3)(2,2.85)(2.1,2)(2.85,1.03)(3,1.02)
\put(1.7,2){$g$}

%de v à x pour g
\psecurve{->}(2,4)(1.9,3.9)(1,2)(2.85,0.97)(3,1)
\put(0.6,2){$g$}

%maintenant le symétrique

%de t à w pour g
\psecurve{->}(3.5,6)(3.8,6)(4.1,5.6)(3.8,5.2)(3.5,5.2)
\put(4.3,5.5){$g$}

%de w à z pour g
\psecurve{->}(3.3,5.2)(3.8,5.1) (4.6,4.4) (4.5,3.8)(3.5,3.7)
\put(4.7,4.5){$g$}

%de t à 1 pour f
\psecurve{->}(3.1,6.1)(2.8,6.1)(2.5,6.5)(2.8,6.9)(3.1,6.9)
\put(2.2,6.4){$f$}

%de w à t pour f
\psecurve{->}(3.1,5.1)(2.9,5.1)(2.8,5.5)(2.9,5.9)(3.1,5.9)
\put(2.5,5.3){$f$}

%de v à t pour f
\psecurve{->}(2,4)(2,4.15)(2.1,5)(2.85,5.97)(3,5.98)
\put(1.7,4.8){$f$}

%de u à t pour f
\psecurve{->}(2,3)(1.9,3.1)(1,5)(2.85,6.03)(3,6)
\put(0.6,4.8){$f$}

\put(6.5,6){$f(0)=0, f(x)=y, f(y)=f(z)=z$}
\put(6.5,5){$f(u)=f(v)=f(w)=t, f(t)=f(1)=1$}
\put(6.5,4){$g(1)=1, g(t)=w, g(w)=g(z)=z$}
\put(6.5,3){$g(u)=g(v)=g(y)=x, g(x)=g(0)=0$}
\end{pspicture}
\caption{An algebra generating a V-variety.}\label{F:Valgebra}
\end{figure}

The algebra $C$ is subdirectly irreducible and its congruence lattice is isomorphic to $V$. The smallest
nontrivial congruence (the monolith) collapses only the pair $(u,v)$. Two other nontrivial congruences are
$\alpha=(0xyz)(uvw)(t)(1)$ and $\beta=(0)(x)(yuv)(zwt1)$. The quotients $C/\alpha$ and $C/\beta$ are both
isomorphic to the simple algebra $B=\set{0,a,b,1}$ depicted in Figure~\ref{F:Simpalgebra}.

\begin{figure}[hbt]
\begin{pspicture}(0,-1)(12,4)
\pslinewidth0.4pt
\pscircle*[linecolor=black](2,0){0.1}
\pscircle*[linecolor=black](2,1){0.1}
\pscircle*[linecolor=black](2,2){0.1}
\pscircle*[linecolor=black](2,3){0.1}
\psline(2,0)(2,3)
\put(2.3,0){$0$}
\put(2.3,1){$a$}
\put(2.3,2){$b$}
\put(2.3,3){$1$}

%de 0 à 0 pour f et g
\psecurve{->}(2.5,0)(2.8,0)(3.1,0.1)(2.8,0.2)(2.5,0.2)
\put(3.3,0){$f,g$}
%de 1 à 1 pour f et g
\psecurve{->}(2.5,3.1)(2.8,3.1)(3.1,3.2)(2.8,3.3)(2.5,3.3)
\put(3.3,3.2){$f,g$}

%de a à b pour f
\psecurve{->}(2.5,1.2)(2.8,1.2)(3.1,1.6)(2.8,2)(2.5,2)
\put(3.3,1.5){$f$}

%de b à 1 pour f
\psecurve{->}(2.5,2.2)(2.8,2.2)(3.1,2.6)(2.8,3)(2.5,3)
\put(3.3,2.5){$f$}

%de a à 0 pour g
\psecurve{->}(2.1,0.9)(1.8,0.9)(1.5,0.5)(1.8,0.1)(2.1,0.1)
\put(1.2,0.4){$g$}

%de b à a pour g
\psecurve{->}(2.1,1.9)(1.8,1.9)(1.5,1.5)(1.8,1.1)(2.1,1.1)
\put(1.2,1.4){$g$}

\put(6,3){$f(0)=0, f(a)=b, f(b)=f(1)=1$}
\put(6,2){$g(1)=1, g(b)=a, g(a)=g(0)=0$}
\end{pspicture}
\caption{A simple quotient of $C$.}\label{F:Simpalgebra}
\end{figure}
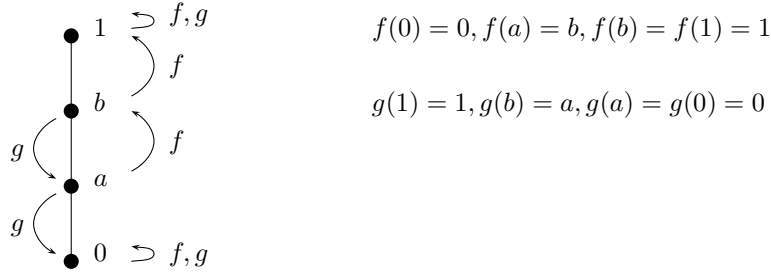

Let $\cC$ be the variety generated by $C$. Thus, $\cC$ is finitely generated and congruence-distributive. Let us check the subdirectly irreducible members of $\cC$. By J\'onsson's lemma, we need to investigate the quotients of subalgebras of $C$. Let $\dv$ denote the
subalgebra $\set{0,1}$ of $C$. Notice that all the two-element subalgebras of $C$ are isomorphic to $\dv$.
The subalgebras of $C$ with more than two elements are $C_1=\set{0,1,z}$, $C_2=\set{0,x,y,z}$,
$C_3=\set{1,t,w,z}$,
$C_4=\set{0,1,x,y,z}$, $C_5=\set{0,1,t,w,z}$, $C_6=C\setminus\set{u,v}$, $C_7=C\setminus\set{u}$,
$C_8=C\setminus\set{v}$, and $C$ itself. Then
$C_1$ is a subdirect product of two copies of $\dv$, $C_2$ and $C_3$ are isomorphic to $B$. Further, 
$C_4$ and $C_5$ are subdirect products of $B$ and $\dv$, while $C_6$, $C_7$ and $C_8$ are subdirect
products of two copies of $B$. 

Hence $C$, $B$, and $\dv$ are (up to isomorphism) the only subdirectly irreducible members of $\cC$,
so $\cC$ satisfies the assumptions of the previous section and we can use Theorem~\ref{main}
to determine whether $\cC$ is congruence FD-maximal. Since $B$ does not have proper automorphisms,
the homomorphisms $h_0,h_1\colon C\to B$ are determined uniquely (up to interchanging) by $h_0(0)=h_0(x)=h_0(y)=h_0(z)=0$, $h_0(u)=h_0(v)=h_0(w)=a$, $h_0(t)=b$, $h_0(1)=1$,
$h_1(1)=h_1(t)=h_1(w)=h_1(z)=1$, $h_1(u)=h_1(v)=h_1(y)=b$, $h_1(x)=a$, $h_1(0)=0$.
Hence,
\[
E=\set{(0,0),(0,a),(0,b),(0,1),(a,b),(a,1),(b,1),(1,1)}.
\]
The pair $(a,b)\in E$ shows that~\ref{xyz}(iii) is not satisfied, as there is no $t\in B$ such that $(a,t)\in E$ and $(t,b)\in E$. So the set $E$ does not allow arbitrarily large sets of $E$-compatible functions. Since the variety $\cC$ does not provide any other candidates for $C$, $B$, $h_0$ and $h_1$, \ref{main}(2) is not satisfied, so $\cC$ is not congruence FD-maximal.

\section{Conclusion}

The characterization of congruence FD-maximal V-varieties helps to decide whether several critical points are finite or not.

Let $\cV$ be a V-variety. Let $\cW$ be a congruence FD-maximal V-variety. It is easy to see that $\crit{\cV}{\cW}\ge\aleph_0$, moreover $\crit{\cW}{\cV}\ge\aleph_0$ if and only if $\cV$ is congruence FD-maximal. For example, denote by $\cV$ the variety generated by the algebra $C$ in Section~\ref{S:nmax} (see Figure~\ref{F:Valgebra}). We have seen that $\cN_5$ is congruence FD-maximal and $\cV$ is not congruence FD-maximal, it follows that $\crit{\cN_5}{\cV}<\aleph_0$.

When both varieties are not congruence FD-maximal, we cannot even decide whether the critical point is finite or not.

\begin{problem}
Given a V-variety $\cV$, which is not congruence FD-maximal. Find a characterization of the finite lattice in $\Con\cV$.
\end{problem}

The lattice $V$ (cf. Figure~\ref{F:V}) is the smallest distributive lattice, which is not a chain, isomorphic to the congruence lattice of a subdirectly irreducible algebra. It would be interesting to study other examples.

\begin{problem}
Characterize finitely generated congruence FD-maximal varieties.
\end{problem}

\end{document}